\documentclass[10pt, letterpaper,reqno]{amsart}
\usepackage{amssymb}
\usepackage{amsmath}
\usepackage{amsfonts, color}
\usepackage{graphicx}
\usepackage{mathrsfs,amssymb, slashed, cite}

\usepackage{hyperref}
\usepackage{cleveref}

\newcommand{\R}{\mathbb{R}}
\newcommand{\C}{\mathbb{C}}

\newcommand{\F}{\mathcal{F}}

\newtheorem{lemma}{Lemma}[section]

\newtheorem{remark}{Remark}[section]

\newtheorem{theorem}{Theorem}[section]
\numberwithin{equation}{section}

\newcommand{\qtq}[1]{\quad\text{#1}\quad}

\begin{document}

\title[Large scale limit]{Large scale limit for a dispersion-managed NLS} 
\author{Jason Murphy}
\address{Department of Mathematics, University of Oregon}
\email{jamu@uoregon.edu}

\maketitle

\begin{abstract} We derive the standard power-type NLS as a scaling limit of the Gabitov--Turitsyn dispersion-managed NLS, using the $2d$ defocusing, cubic equation as a model case. In particular, we obtain global-in-time scattering solutions to the dispersion-managed NLS for large scale data of arbitrary $L^2$-norm. 
\end{abstract}

\section{Introduction}

Dispersion-managed nonlinear Schr\"odinger equations comprise a class of partial differential equations arising in the setting of nonlinear fiber optics.  Here `dispersion management' refers to the periodic variation of dispersion along the fiber, which may be modeled via
\[
i\partial_t u + \gamma(t)\Delta u = |u|^2 u,
\] 
with $\gamma:\R\to\R$ a (typically piecewise constant) 1-periodic function \cite{Agrawal, Kurtzke}. Such equations have been the topic of recent mathematical interest (see e.g.  \cite{Agrawal, Kurtzke,CL, Lushnikov, GT1, HL, PZ, ZGJT, CHL, CL, CMVH, MVH, Murphy, AntonelliSautSparber}). In particular, there has been some recent interest in averaging/homogenization phenomena in the setting of dispersion-managed NLS, in which different effective equations are derived under certain limiting regimes.   Without going into the details here, in the `\emph{strong dispersion-management}' regime, one can derive the `Gabitov--Turitsyn' equation
\begin{equation}\label{dmnls}
(i\partial_t +\Delta) u = \int_0^1 e^{-i\sigma\Delta}\bigl[|e^{i\sigma\Delta}u|^2 e^{i\sigma\Delta}u\bigr]\,d\sigma,
\end{equation}
while in the `\emph{fast dispersion-management}' regime, one can derive the standard power-type NLS 
\begin{equation}\label{nls}
(i\partial_t + \Delta) u = |u|^2 u
\end{equation}
(see e.g. \cite{AntonelliSautSparber, CMVH, CL, Murphy, ZGJT}).  In fact, although it is more natural to derive the \emph{focusing} versions of these equations for applications in nonlinear optics, we will restrict to the \emph{defocusing} case throughout this paper (for reasons to be explained shortly).

Given the revelance of \eqref{dmnls} in the nonlinear optics setting (particularly in the investigation of optical solitons), there has been recent interest in obtaining a better mathematical understanding of this model. The contribution of the present note is to show that one can derive the standard power-type NLS \eqref{nls} as a scaling limit of the dispersion-managed NLS \eqref{dmnls} (specifically, in the large scale regime).  To minimize technical complications, we work exclusively with the defocusing, $2d$ cubic model, so that the underlying NLS admits scattering solutions for arbitrary data in $L^2$ (cf. \cite{Dodson}).

The basic idea of the proof is that if an NLS solution has a large characteristic spatial scale (and thus, by the scaling symmetry of \eqref{nls}, a large characteristic timescale), then the effect of the averaging in \eqref{dmnls} should be negligible.  In particular, large scale solutions of NLS will approximately solve \eqref{dmnls}; thus, using a stability result for \eqref{dmnls}, we will be able to construct solutions to \eqref{dmnls} that are close to the NLS solutions.  

Our main result may be stated as follows: 

\begin{theorem}\label{T} Let $\varphi \in L^2(\R^2)$, and let $u$ be the global, scattering solution to \eqref{nls} with $u|_{t=0}=\varphi$. For $\lambda>0$, let
\[
u_\lambda(t,x):=\lambda u(\lambda^2 t,\lambda x),\qtq{with initial data}\varphi_\lambda(x) = \lambda\varphi(\lambda x).
\]
For all $\lambda$ sufficiently small, the solution $v_\lambda$ to \eqref{dmnls} with initial data $\varphi_\lambda$ exists globally in time and scatters\footnote{We say a solution $v$ \emph{scatters} in $L^2$ if there exists $v_\pm\in L^2$ such that \[\lim_{t\to\pm\infty}\|v(t)-e^{it\Delta}v_\pm\|_{L^2}=0.\]} in $L^2$.  Moreover, we have
\begin{equation}\label{main-result}
\lim_{\lambda\to 0}\|u_\lambda - v_\lambda\|_{L_t^\infty L_x^2 \cap L_{t,x}^4(\R\times\R^2)} = 0. 
\end{equation}
\end{theorem}

As an application of Theorem~\ref{T}, we deduce the existence of global, scattering solutions to \eqref{dmnls} of arbitrarily large $L^2$-norm.   Specifically, any $L^2$ function $\varphi$ rescaled to a sufficiently large spatial scale leads to a scattering solution to \eqref{dmnls}.  While it has been previously shown in \cite{ChoiLeeLeeScattering, KawakamiMurphy} that functions of sufficiently small $L^2$-norm lead to scattering solutions, the large-data scattering problem for this model is still open in general.  We also remark that the limit in \eqref{main-result} does not follow from a quantitative estimate in $\lambda$.  It is possible that this estimate could be made quantitative if one worked with stronger assumptions on the data $\varphi$.

The arguments given here would apply equally well using scattering solutions to the focusing equation.  However, for a non-scattering solution to the underlying equation (e.g. a soliton), the arguments here would only provide approximation for finite times. 

As mentioned above, the idea of the proof is to show that the rescaled solutions to \eqref{nls} approximately solve \eqref{dmnls} in the `large scale' regime and then apply a stability result for \eqref{dmnls} (Theorem~\ref{T:stability} below). The key decomposition in the proof appears in \eqref{low}--\eqref{dm-high} below and is explained there.  The argument is inspired by the work \cite{Ntekoume} on homogenization for the cubic NLS.  This type of argument was also used in \cite{CMVH} to prove averaging for the dispersion-managed NLS.

It is also natural to ask about small-scale limit of \eqref{dmnls}.  Indeed, it is likely that the small-scale limit of \eqref{dmnls} will be important in order to understand blowup phenomena for the dispersion-managed NLS, or even to establish scattering (e.g. using concentration-compactness methods).  A natural candidate for the small-scale limit of \eqref{dmnls} is
\begin{equation}\label{smallscale}
(i\partial_t + \Delta) u = \int_0^\infty e^{-i\sigma\Delta}(|e^{i\sigma\Delta} u|^2 e^{i\sigma\Delta}u)\,d\sigma,
\end{equation}
or perhaps the same equation with the integration taken over $\R$.   Informally, if $u$ has a very short characteristic timescale, then averaging over $\sigma\in[0,1]$ should not be so different from averaging over $\sigma\in[0,\infty)$.

In fact, the latter equation (with integration over $\R$) was already been mentioned in \cite{ChoiHongLee}.  In that work, the authors constructed a soliton solution to the focusing version of this equation (in $1d$ with sufficiently high power nonlinearity) and showed that this solution may be used to define a threshold below which one has a dichotomy between global existence and blowup for the Gabitov--Turitsyn DMNLS.  

In this paper, we focus only on the large-scale limit, which only requires a good understanding of equations \eqref{dmnls} and \eqref{nls}.  We expect that treating the small-scale limit will first require more analysis of the model \eqref{smallscale}.  This will be the topic of future work. 

The rest of this paper is organized as follows:  Section~\ref{S:prelim} introduces notation and collects some preliminary results.  Section~\ref{S:stability} concerns well-posedness and stability for the DMNLS \eqref{dmnls}.  Finally, Section~\ref{S:convergence} contains a proof of the main result, Theorem~\ref{T}. 

\subsection*{Acknowledgements} J. M. was supported by NSF grant DMS-2350225 and Simons Foundation grant MPS-TSM-00006622.  We are grateful to J. Kawakami for helpful discussions related to this work.

\section{Preliminaries}\label{S:prelim}

We use the standard $\lesssim$ notation, as well as the standard notation for mixed Lebesgue space-time norms.  In fact, we will need to consider norms of the form
\[
\| F\|_{L_\theta^a L_t^b L_x^c(I\times J\times\R^d)} := \bigl\| \ \bigl\| \, \|F(\theta,t,x)\|_{L_x^c(\R^d)} \bigr\|_{L_t^b(J)} \bigr\|_{L_\theta^a(I)}. 
\]
The free Schr\"odinger propagator is denoted by $e^{it\Delta}$.  It is the Fourier multiplier operator defined via $e^{it\Delta}=\F^{-1}e^{-it|\xi|^2}\F$. 

We employ the standard Littlewood--Paley frequency decomposition in Section~\ref{S:convergence}.  In particular, we write $u_{\leq N} =P_{\leq N}u$, where $P_{\leq N}$ is (on the Fourier side) a smooth cutoff to $\{|\xi|\leq N\}$.  Similarly, we have $u_{>N}=P_{>N}u = u-P_{\leq N} u$. We utilize standard estimates associated to these operators, e.g. the Bernstein estimates
\[
\|\nabla P_{\leq N}\varphi\|_{L_x^r} \lesssim N \|\varphi\|_{L_x^r},\quad \|P_{>N}\varphi\|_{L_x^r} \lesssim N^{-1}\|\nabla\varphi\|_{L_x^r}
\]
for $r\in[1,\infty]$. 

The work \cite{KawakamiMurphy} utilized a so-called `shifted Strichartz estimate' (inspired by the methods of \cite{Kawakami}), which can be used together with function spaces defined via $e^{i\theta\Delta}u\in L_\theta^\infty L_t^q L_x^r$ to deal efficiently with the nonlinear structure in \eqref{dmnls}. We record the particular estimate we need here. 

\begin{theorem}[Shifted Strichartz estimate, \cite{KawakamiMurphy}] Let $I\ni 0$ be an interval. Then we have
\[
\biggl\| \int_0^t e^{i(t-s+\theta-\sigma)\Delta}F(\sigma,s)\,ds\biggr\|_{L_{t,x}^4(I\times\R^d)} \lesssim \|F(\sigma,t)\|_{L_{t,x}^{\frac{4}{3}}(I\times\R^d)} 
\]
uniformly in $\theta,\sigma\in\R$. 
\end{theorem}

\begin{proof} While this estimate was proven in more generality in \cite{KawakamiMurphy}, we will include a proof here for the sake of completeness. Note that for $t\in I$ and $s\in[0,t]$ we may write
\[
F(\sigma,s)=\tilde F(\sigma,s):=\chi_I(s)F(\sigma,s). 
\]
We argue by duality.  Fix $G\in L_{t,x}^{\frac43}$ of norm one.  By the dispersive estimate for the free propagator and H\"older's inequality, we have (for any $\theta,\sigma$):
\begin{align*}
\int_I & \int_0^t \int |e^{i(t-s+\theta-\sigma)\Delta}\tilde F(\sigma,s)|\,|G(t,x)|\,dx\,ds\,dt \\
& \lesssim \int_\R \int_\R |t-s+\theta-\sigma|^{-\frac12}\|\tilde F(\sigma,s)\|_{L_x^{\frac43}}\|G(t)\|_{L_x^{\frac43}} \,ds\,dt \\
& \lesssim \int_\R \bigl[K\ast \|\tilde F(\sigma,\cdot)\|_{L_x^{\frac43}}\bigr](t+\theta-\sigma)\|G(t)\|_{L_x^{\frac43}}\,dt, 
\end{align*}
where $K(s)=|s|^{-\frac12}$. Thus, by translation invariance of Lebesgue measure and Hardy--Littlewood--Sobolev, we have
\begin{align*}
\int_I & \int_0^t \int |e^{i(t-s+\theta-\sigma)\Delta}\tilde F(\sigma,s)|\,|G(t,x)|\,dx\,ds\,dt \\
&\lesssim \| K\ast \|\tilde F(\sigma,\cdot)\|_{L_x^{\frac43}}\|_{L_t^4} \|G\|_{L_{t,x}^{\frac43}}\\
& \lesssim \||t|^{-\frac12}\|_{L_t^{2,\infty}} \|F(\sigma,t)\|_{L_{t,x}^{\frac43}(I\times\R^2)} \\
& \lesssim \|F(\sigma,t)\|_{L_{t,x}^{\frac43}(I\times\R^2)},
\end{align*}
as desired. 
\end{proof}

\section{Stability Theory for DMNLS}\label{S:stability}

The goal of this section is Theorem~\ref{T:stability} below, which is a stability result for \eqref{dmnls}.  After introducing the proper function spaces, the proof follows along standard lines (see e.g. \cite{Clay}), utilizing the shifted Strichartz estimate as appropriate.

Throughout this section, we write
\[
F(\psi) = |\psi|^2 \psi \qtq{and} F_{DM}(\psi) = \int_0^1 e^{-i\sigma\Delta}\bigl[ F(e^{i\sigma\Delta}\psi)\bigr]\,d\sigma. 
\]

We first establish local existence in $L^2$ for \eqref{dmnls}. 

\begin{theorem}[Local existence] For any $\varphi\in L^2$, there exists $T>0$ and a unique solution $u:(-T,T)\times\R^2\to\C$ to \eqref{dmnls} with $u|_{t=0}=\varphi$. The solution $u$ may be extended in time as long as the norm
\[
\|e^{i\theta\Delta}u\|_{L_\theta^\infty L_{t,x}^4}
\]
remains finite.
\end{theorem}

\begin{proof} We follow the usual approach to the `critical' local well-posedness theory.  As the proof follows mostly via standard arguments (incorporating the shifted Strichartz estimate in place of the standard Strichartz estimate as needed), we will keep the presentation brief.  We refer the reader to \cite{Clay} for a more thorough presentation of such arguments.  We also work forward in time only for the sake of simplicity.  

Fix $\varphi\in L^2$.  To get the argument started, we need the following:
\begin{equation}\label{small-data}
\forall\eta>0\ \exists T>0\qtq{such that}\|e^{i\theta\Delta}e^{it\Delta}\varphi\|_{L_{\theta}^\infty L_{t,x}^4([0,1]\times(-T,T)\times\R^2)} < \eta.
\end{equation} 
In the standard case (without the factor $e^{i\theta\Delta}$ and the $L_\theta^\infty$-norm), this follows immediately from Strichartz estimates and the monotone convergence theorem.  We will be able to obtain \eqref{small-data} by exploiting the fact that $\theta$ is restricted to a compact set.

Indeed, suppose \eqref{small-data} fails.  Then (changing variables in $t$) we may find $\eta_0>0$ and a sequence of $\theta_n\in[0,1]$ such that
\[
\|e^{it\Delta}\varphi\|_{L_{t,x}^4((\theta_n-\frac{1}{n},\theta_n+\frac{1}{n})\times\R^2)}\geq \eta_0\qtq{for all}n.
\]
Passing to a subsequence, we may assume that $\theta_n\to\theta_*\in[0,1]$.  In this case, we deduce that
\[
\|e^{it\Delta}\varphi\|_{L_{t,x}^4((\theta_*-\frac{2}{n},\theta_*+\frac{2}{n})\times\R^2)}\geq \eta_0\qtq{for all large}n. 
\]
However, recalling that $e^{it\Delta}\varphi\in L_{t,x}^4(\R\times\R^2)$ (by the Strichartz estimate), this contradicts the monotone convergence theorem.

We now let $\eta>0$ be a small parameter and choose $T$ as in \eqref{small-data}.  We now wish to show that the map
\[
\Phi u(t) = e^{it\Delta}\varphi - i\int_0^t e^{i(t-s)\Delta}\int_0^1 e^{-i\sigma\Delta}(|e^{i\sigma\Delta}u|^2 e^{i\sigma\Delta} u)\,d\sigma\,ds 
\]
is a contraction on the space 
\[
X=\{u: \|e^{i\theta\Delta}u\|_{L_\theta^\infty L_{t,x}^4} \leq 2\eta\}
\]
with distance given by
\[
d(u,v)=\|e^{i\theta\Delta}[u-v]\|_{L_\theta^\infty L_{t,x}^4}.
\]
Here all norms are taken over $[0,1]\times[0,T]\times\R^2$ unless indicated otherwise.  The estimates below should make it clear that we could also include the condition $\|u\|_{L_t^\infty L_x^2}\leq 2C\|\varphi\|_{L_x^2}$ in the definition of $X$, but for the sake of simplicity we will omit this norm from consideration here. 

For $u\in X$, we now use the shifted Strichartz estimate to obtain
\begin{align*}
\|&e^{i\theta\Delta}\Phi u\|_{L_\theta^\infty L_{t,x}^4}\\
 & \lesssim \|e^{i\theta\Delta}e^{it\Delta}\varphi\|_{L_\theta^\infty L_{t,x}^4} + \biggl\|\int_0^t\int_0^1 e^{i(t-s+\theta-\sigma)\Delta}(|e^{i\sigma\Delta}u|^2 e^{i\sigma\Delta} u)\,d\sigma\,ds\biggr\|_{L_{\theta}^\infty L_{t,x}^4} \\
& \lesssim \eta + \int_0^1 \biggl\| \int_0^t e^{i(t-s+\theta-\sigma)\Delta}(|e^{i\sigma\Delta}u|^2 e^{i\sigma\Delta} u)\,ds\biggr\|_{L_\theta^\infty L_{t,x}^4} \,d\sigma\\
& \lesssim \eta + \int_0^1 \| e^{i\sigma\Delta} u\|_{L_{t,x}^{4}}^3\,d\sigma  \lesssim \eta + \|e^{i\sigma\Delta}u\|_{L_\sigma^\infty L_{t,x}^4}^3  \lesssim \eta + \eta^3, 
\end{align*}
so that (for $\eta$ small) we obtain $\Phi:X\to X$.

Taking $u,v\in X$, similar estimates lead to
\begin{align*}
\|e^{i\theta\Delta}&[\Phi u - \Phi v]\|_{L_\theta^\infty L_{t,x}^4} \\
& \lesssim \int_0^1 \| |e^{i\sigma\Delta}u|^2 e^{i\sigma\Delta} u - |e^{i\sigma\Delta}v|^2 e^{i\sigma\Delta} v\|_{L_{t,x}^{\frac43}}\,d\sigma \\
& \lesssim \{\|e^{i\sigma\Delta}u\|_{L_{\sigma}^\infty L_{t,x}^4}^2 + \|e^{i\sigma\Delta}u\|_{L_\sigma^\infty L_{t,x}^4}^2\}\|e^{i\theta\Delta}[u-v]\|_{L_\theta^\infty L_{t,x}^4} \\
& \lesssim \eta \|e^{i\theta\Delta}[u-v]\|_{L_\theta^\infty L_{t,x}^4},
\end{align*}
so that $\Phi$ is a contraction for $\eta$ small.  This yields the existence of the solution $u$.  The estimates above also show that $e^{i\theta\Delta}u\in L_\theta^\infty L_{t,x}^4$ is a controlling norm for this equation, so that as long as this norm remains finite we can continue our solution in time.  \end{proof}

We turn to the stability results for \eqref{dmnls}, beginning with a short-time result. 

\begin{lemma}[Short-time stability]\label{L:short} Suppose $u$ satisfies
\[
\begin{cases} 
(i\partial_t + \Delta) u = F_{DM}(u) + G, \\
u|_{t=0}=u_0\in L^2
\end{cases}
\]
on an interval $I$ with $\inf I = 0$.  Assume that the following estimates hold:
\begin{align*}
&\|e^{i\theta\Delta}u\|_{L_{\theta}^\infty L_{t,x}^4([0,1]\times I\times\R^2)}< \eta_0, \\
&\|e^{i(t+\theta)\Delta}[u_0-v_0]\|_{L_{\theta}^\infty L_{t,x}^4([0,1]\times I\times\R^2)}<\eta, \\
& \biggl\| \int_0^t e^{i(t+\theta-s)\Delta} G(s)\,ds\biggr\|_{L_{\theta}^\infty L_{t,x}^4([0,1]\times I\times\R^2)}<\eta
\end{align*}
for some $v_0\in L^2$ and $0<\eta<\eta_0$. If $\eta_0$ is sufficiently small, then the solution $v$ to \eqref{dmnls} with $v|_{t=0}=v_0$ exists on $I$ and the following estimates hold:
\begin{align*}
&\|e^{i\theta\Delta}(u-v)\|_{L_{\theta}^\infty L_{t,x}^4([0,1]\times I\times\R^2)}\lesssim \eta,\\
& \|F(e^{i\theta\Delta}u)-F(e^{i\theta\Delta}v)\|_{L_\theta^\infty L_{t,x}^{\frac43}([0,1]\times I \times\R^2)}\lesssim \eta. 
\end{align*}
\end{lemma}

\begin{proof} We let $v$ be the solution to \eqref{dmnls} with $v|_{t=0}=u_0$.  The solution may be extended as long as we control the $L_{\theta}^\infty L_{t,x}^4$-norm.  Thus it suffices to assume the existence of $v$ on the full interval $I$ and establish the desired estimates as \emph{a priori} estimates. 

We define $w=u-v$, so that
\[
\begin{cases} (i\partial_t+\Delta) w = F_{DM}(w+u)-F_{DM}(u) + G, \\
w|_{t=0}=u_0-v_0. 
\end{cases}
\]
For $T\in I$, let us define
\[
E(T) = \|F(e^{i\sigma\Delta}(w+u)) - F(e^{i\sigma\Delta}u)\|_{L_\sigma^\infty L_{t,x}^{\frac43}([0,1]\times[0,T]\times\R^2)}. 
\]
Using the assumptions on $u_0-v_0$ and $G$ and the shifted Strichartz estimate, we claim that
\begin{align*}
\|e^{i\theta\Delta} &w\|_{L_{\theta}^\infty L_{t,x}^4([0,1]\times[0,T]\times\R^2)} \\
& \lesssim \eta + \biggl\|\int_0^t e^{i(t+\theta-s)\Delta}[F_{DM}(w+u)-F_{DM}(u)]\,ds\biggr\|_{L_{\theta}^\infty L_{t,x}^4([0,1]\times [0,T]\times\R^2)} \\
& \lesssim \eta + E(T). 
\end{align*}
Indeed, to pass from the second to third line above, we estimate as follows: by the shifted Strichartz estimate and H\"older's inequality, 
\begin{equation}\label{FDM_demo}
\begin{aligned}
 \biggl\|&\int_0^t e^{i(t+\theta-s)\Delta}[F_{DM}(w+u)-F_{DM}(u)]\,ds\biggr\|_{L_{\theta}^\infty L_{t,x}^4([0,1]\times [0,T]\times\R^2)} \\
 & \lesssim\int_0^1 \biggl\| \int_0^t e^{i(t-s+\theta-\sigma)\Delta}[F(e^{i\sigma\Delta}(w+u))-F(e^{i\sigma\Delta}u)]\,ds\biggr\|_{L_{\theta}^\infty L_{t,x}^4([0,1]\times [0,T]\times\R^2)}\,d\sigma \\
 & \lesssim \int_0^1 \| F(e^{i\sigma\Delta}(w+u))-F(e^{i\sigma\Delta}u)\|_{L_\theta^\infty L_{t,x}^{\frac43}([0,1]\times[0,T]\times\R^2)}\,d\sigma \\
 & \lesssim E(T).
\end{aligned}
\end{equation}

To estimate $E(T)$, we use H\"older's inequality, the assumptions on $u$, and the bound on $e^{i\theta\Delta}w$ above to obtain
\begin{align*}
E(T) & \lesssim \|e^{i\sigma\Delta}w\|_{L_{\sigma}^\infty L_{t,x}^4([0,1]\times[0,T]\times\R^2)}^3 + \|e^{i\sigma\Delta}u\|_{L_{\sigma}^\infty L_{t,x}^4}^2 \|e^{i\sigma\Delta}w\|_{L_{\sigma}^\infty L_{t,x}^4([0,1]\times[0,T]\times\R^2)} \\
& \lesssim (\eta+E(T))^3 + \eta_0^2(\eta + E(T)). 
\end{align*}
As $E(0)=0$, a continuity argument implies that for $\eta_0$ sufficiently small, we have $E(T)\lesssim\eta$ for all $t\in I$. Recalling the bound on $e^{i\theta\Delta}w$ above, this leads to both of the desired estimates. \end{proof}

Finally, we prove our main stability result for \eqref{dmnls}.  It is obtained by splitting into intervals on which the norm of the approximate solution is small and iterating the short-time result. 

\begin{theorem}[Long-time stability]\label{T:stability}
Suppose $u$ satisfies
\[
\begin{cases} 
(i\partial_t + \Delta) u = F_{DM}(u) + G, \\
u|_{t=0}=u_0\in L^2
\end{cases}
\]
on an interval $I$ with $\inf I = 0$.  Assume that the following estimates hold:
\begin{align*}
&\|e^{i\theta\Delta}u\|_{L_{\theta}^\infty L_{t,x}^4([0,1]\times I\times\R^2)}\leq L, \\
&\|e^{i(t+\theta)\Delta}[u_0-v_0]\|_{L_{\theta}^\infty L_{t,x}^4([0,1]\times I\times\R^2)}<\eta, \\
& \biggl\| \int_0^t e^{i(t+\theta-s)\Delta} G(s)\,ds\biggr\|_{L_{\theta}^\infty L_{t,x}^4([0,1]\times I\times\R^2)}<\eta
\end{align*}
for some $v_0\in L^2$ and $0<\eta<\eta_0$. If $\eta_0=\eta_0(L)$ is sufficiently small, then the solution $v$ to \eqref{dmnls} with $v|_{t=0}=v_0$ exists on $I$ and the following estimate holds:
\begin{align*}
&\|e^{i\theta\Delta}(u-v)\|_{L_{\theta}^\infty L_{t,x}^4([0,1]\times I\times\R^2)}\lesssim \eta.
\end{align*}

\end{theorem}

\begin{proof} Let $\eta_0$ be a small parameter and split $I$ into $J=J(L)$ intervals $I_j=[t_j,t_{j+1}]$ such that
\[
\|e^{i\theta\Delta}u\|_{L_{\theta}^\infty L_{t,x}^4([0,1]\times I_j\times\R^2)} < \eta_0 \qtq{for each}j. 
\]
We will prove by induction that for each $j=0,\dots,J-1$, we have
\begin{align*}
&\|e^{i\theta\Delta}(u-v)\|_{L_{\theta}^\infty L_{t,x}^4([0,1]\times I_j\times\R^2)} \leq C_j\eta, \\
&\|F(e^{i\theta\Delta}u)-F(e^{i\theta\Delta}v)\|_{L_\theta^\infty L_{t,x}^{\frac43}([0,1]\times I_j\times\R^2)} \leq C_j\eta. 
\end{align*}

The base case follows immediately from Lemma~\ref{L:short}.  Now assume the desired estimates hold for $0\leq k\leq j-1$.  We wish to apply Lemma~\ref{L:short} on the interval $I_j=[t_j,t_{j+1}]$.  Using the same notation $w=u-v$ as in the proof of Lemma~\ref{L:short}, this requires that we verify 
\[
\|e^{i(t+\theta-t_j)\Delta}w(t_j)\|_{L_{\theta}^\infty L_{t,x}^4([0,1]\times I_j\times\R^2)} \lesssim \eta. 
\]
We use the Duhamel formula to write
\begin{align*}
e^{i(t+\theta-t_j)\Delta}w(t_j) = e^{i(t+\theta)\Delta}w_0 - i\int_0^{t_j} e^{i(t+\theta-s)\Delta}[F_{DM}(u+w)-F_{DM}(u)-G]\,ds. 
\end{align*}
Recalling the assumptions on $u_0-v_0$ and $G$, it follows that 
\begin{align*}
\|&e^{i(t+\theta-t_j)\Delta}w(t_j) \|_{L_{\theta}^\infty L_{t,x}^4([0,1]\times I_j\times \R^2} \\
& \lesssim \eta + \biggl\| \int_0^{t_j} e^{i(t+\theta-s)\Delta}[F_{DM}(u+w)-F_{DM}(u)]\,ds\biggr\|_{L_{\theta}^\infty L_{t,x}^4([0,1]\times I_j\times\R^2)}. 
\end{align*}

Estimating as we did to obtain \eqref{FDM_demo} and using the inductive hypothesis, we find that
\begin{align*}
 \biggl\| &\int_0^{t_j} e^{i(t+\theta-s)\Delta}[F_{DM}(u+w)-F_{DM}(u)]\,ds\biggr\|_{L_{\theta}^\infty L_{t,x}^4([0,1]\times I_j\times\R^2)} \\
 & \lesssim \int_0^1 \|[F(e^{i\sigma\Delta}(w+u))-F(e^{i\sigma\Delta}u)]\chi_{[0,t_j]}\|_{L_{t,x}^{\frac43}([0,t_{j+1}]\times\R^2)} \,d\sigma \\
 & \lesssim \sum_{k=0}^{j-1} \|F(e^{i\sigma\Delta}(w+u))-F(e^{i\sigma\Delta}u)\|_{L_\sigma^\infty L_{t,x}^{\frac43}([0,1]\times I_k\times\R^2)} \\
 & \lesssim \sum_{k=0}^{j-1}C_k\eta. 
\end{align*}

Choosing $\eta_0=\eta_0(J)$ sufficiently small, we find that the hypotheses of Lemma~\ref{L:short} will be satisfied, and an application of Lemma~\ref{L:short} on the interval $I_j$ then yields the desired conclusions on the interval $I_j$. \end{proof}

\begin{remark} One can also propagate estimates in $L_t^\infty L_x^2$. That is, if we further assume that
\[
\|u_0-v_0\|_{L_x^2} < \eta\qtq{and}\|u\|_{L_t^\infty L_x^2}\lesssim 1 
\]  
and that the bound on $G$ holds in $L_t^\infty L_x^2$, then we can obtain
\[
\|u-v\|_{L_t^\infty L_x^2} \lesssim \eta.
\]
\end{remark}

\section{Convergence}\label{S:convergence}

This section contains the proof of the main result, Theorem~\ref{T}.

\begin{proof}[Proof of Theorem~\ref{T}] We let $u$ be the solution to \eqref{nls} with $u|_{t=0}=\varphi\in L^2$.  We define
\[
u_\lambda(t,x)= S_\lambda u (t,x)   := \lambda u(\lambda^2 t,\lambda x).
\]
We wish to show that $u_\lambda$ define approximate solutions to \eqref{dmnls} in the sense required by Theorem~\ref{T:stability}. 

We first need to show that the $u_\lambda$ satisfy the necessary (averaged) space-time bounds.  We begin by writing
\begin{equation}\label{decomp_with_prop}
e^{i\theta\Delta}u_\lambda(t) = e^{i(t+\theta)\Delta}\varphi_\lambda - i \int_0^t e^{i(t+\theta-s)\Delta}|u_\lambda|^2 u_\lambda\,ds. 
\end{equation}
Thus using the shifted Strichartz estimate (and taking all norms over $[0,1]\times\R\times\R^2$),
\begin{equation}\label{av-bd}
\begin{aligned}
\|e^{i\theta\Delta}u_\lambda\|_{L_{\theta}^\infty L_{t,x}^4} & \lesssim \|e^{i(t+\theta)\Delta}\varphi_\lambda\|_{L_{\theta}^\infty L_{t,x}^4} + \biggl\| \int_0^t e^{i(t+\theta-s)\Delta}|u_\lambda|^2 u_\lambda\,ds\biggr\|_{L_{\theta}^\infty L_{t,x}^4} \\
& \lesssim \| e^{it\Delta}\varphi_\lambda\|_{L_{\theta}^\infty L_{t,x}^4} + \| \,|u_\lambda|^2 u_\lambda\|_{L_\theta^\infty L_{t,x}^{\frac43}} \\
& \lesssim \|e^{it\Delta}\varphi_\lambda\|_{L_{t,x}^4} + \|\,|u_\lambda|^2 u_\lambda\|_{L_{t,x}^{\frac43}} \\ 
& \lesssim \|\varphi\|_{L_x^2} + \|u\|_{L_{t,x}^4}^3 \lesssim 1
\end{aligned}
\end{equation}
uniformly in $\lambda$.  Similarly, using unitarity of the free propagator in $L^2$ and the standard Strichartz estimate, 
\begin{align*}
\|u_\lambda\|_{L_t^\infty L_x^2} & \lesssim \|\varphi\|_{L_x^2} + \biggl\| \int_0^t e^{-is\Delta}|u_\lambda|^2 u_\lambda\,ds\biggr\|_{L_t^\infty L_x^2} \\
& \lesssim \|\varphi\|_{L_x^2} + \|\,|u_\lambda|^2 u_\lambda \|_{L_{t,x}^{\frac43}} \\
& \lesssim \|\varphi\|_{L_x^2} + \|u\|_{L_{t,x}^4}^3  \lesssim 1
\end{align*}
uniformly in $\lambda$.

As before, we write
\[
F(\psi) = |\psi|^2 \psi \qtq{and} F_{DM}(\psi) = \int_0^1 e^{-i\sigma\Delta}\bigl[ F(e^{i\sigma\Delta}\psi)\bigr]\,d\sigma. 
\]

We wish to prove
\begin{equation}\label{WTS}
\lim_{\lambda\to 0} \,\biggl\| \int_0^t e^{i(t+\theta-s)\Delta}[(i\partial_s + \Delta)u_\lambda - F_{DM}(u_\lambda) ]\,ds\biggr\|_{L_{\theta}^\infty L_{t,x}^4([0,1]\times\R\times\R^2)}=0. 
\end{equation}

To this end, let $\eta>0$ and $N\in 2^{\mathbb{Z}}$ to be chosen below.  Using the scaling symmetry for \eqref{nls} and recalling the notation for Littlewood--Paley projections, we may write 
\begin{align}
(i\partial_t + \Delta)u_\lambda - F_{DM}(u_\lambda) & = F(S_\lambda u_{\leq N}) - F_{DM}(S_\lambda u_{\leq N}) \label{low} \\
&\quad + F(S_\lambda u) - F(S_\lambda u_{\leq N}) \label{high}\\
&\quad - [F_{DM}(S_\lambda u)-F_{DM}(S_\lambda u_{\leq N})].\label{dm-high}
\end{align}

The plan is as follows:  In the term \eqref{low}, we will find a cancellation due to the structure of the nonlinearities in the regime $\lambda\to 0$.  However, to exploit this cancellation will require that we estimate higher order derivatives of the solution.  Thus we apply this argument only to the low frequency part of the solution.  The remaining terms \eqref{high} and \eqref{dm-high} will contain a copy of the high frequency part of the solution, which we may make small by choosing the frequency sufficiently large.  In particular, we will first choose $N=N(\eta)$ sufficiently large, and then $\lambda=\lambda(N)$ sufficiently small in the argument below.  We remark that this approach parallels that of \cite{CMVH, Ntekoume} (with the original argument appearing in \cite{Ntekoume}). 

We turn to the details, beginning with the high frequency estimates. We first claim that
\begin{equation}\label{high-control1}
\|u_{>N}\|_{L_{t,x}^4(\R\times\R^2)}\lesssim \eta \qtq{for all}N\qtq{sufficiently large.} 
\end{equation}
We argue as in \cite[Proof of (4.6)]{CMVH}.  We let $N_0$ be large enough that 
\begin{equation}\label{N_0}
\|P_{>N_0}\varphi\|_{L^2}<\eta.
\end{equation}
We let $w$ be the solution to \eqref{nls} with data $P_{\leq N_0}\varphi$.  By scattering and persistence of regularity for \eqref{nls}, the solution $w$ is global in time and satisfies
\[
\|\nabla w\|_{L_{t,x}^4}\lesssim N_0. 
\]
Moreover, since the difference of the initial data of $u$ and $v$ is $P_{>N_0}\varphi$, we have by the stability theory for \eqref{nls} that 
\[
\|u-w\|_{L_{t,x}^4} \lesssim \eta
\]
(for sufficiently small $\eta$). Thus, using Bernstein estimates,
\[
\|u_{>N}\|_{L_{t,x}^4} \lesssim \|u-w\|_{L_{t,x}^4} + N^{-1}\|\nabla w\|_{L_{t,x}^4} \lesssim \eta+N^{-1} N_0.
\]

This yields \eqref{high-control1} for $N\gtrsim N_0\eta^{-1}$.  In fact, we will choose our $N$ large enough that \eqref{high-control1} holds with the choice $N/8$, as well. 

We further claim that
\begin{equation}\label{high-control3}
\|e^{i\theta\Delta}S_\lambda u_{>N}\|_{L_{\theta}^\infty L_{t,x}^4([0,1]\times\R\times\R^2)}\lesssim \eta \qtq{uniformly in}\lambda>0,
\end{equation}
which will be needed to estimate \eqref{dm-high}. To prove this, first note that by scaling, we have
\[
\|e^{i\theta\Delta}S_\lambda u_{>N}\|_{L_{\theta}^\infty L_{t,x}^4([0,1]\times\R\times\R^2)} =\|e^{i\theta\Delta}u_{>N}\|_{L_{\theta}^\infty L_{t,x}^4([0,\lambda^2]\times\R\times\R^2)}.
\]
We now proceed as we did in \eqref{decomp_with_prop}, writing 
\begin{align*}
e^{i\theta\Delta}u_{>N}(t) & = e^{i(t+\theta)\Delta}\varphi_{>N} - iP_{>N} \int_0^t e^{i(t+\theta-s)\Delta}|u|^2 u\,ds \\
& = e^{i(t+\theta)\Delta}\varphi_{>N} - iP_{>N} \int_0^t e^{i(t+\theta-s)\Delta}\O(u^2 u_{>\frac{N}{8}})\,ds,
\end{align*}
where we have used the fact that $P_{>N}$ commutes with the free propagator and $P_{>N}F(u_{\leq\frac{N}{8}})\equiv 0$.  Here $\O(u^2 u_{>\frac{N}{8}})$ represents a linear combination of cubic terms in $u$ and $\bar u$, at least one of which is projected to high frequencies. Thus, applying Strichartz estimates, 
\begin{align*}
&\|e^{i\theta\Delta}u_{>N}\|_{L_{\theta}^\infty L_{t,x}^4([0,\lambda^2]\times \R\times\R^2)} \\
& \lesssim \|e^{it\Delta}\varphi_{>N}\|_{L_{\theta}^\infty L_{t,x}^4[0,\lambda^2]\times \R\times\R^2)} + \| u^2 u_{>\frac{N}{8}}\|_{L_{\theta}^\infty L_{t,x}^{\frac{4}{3}}([0,\lambda^2]\times\R\times\R^2)} \\
& \lesssim \|\varphi_{>N}\|_{L_x^2} + \|u\|_{L_{t,x}^4}^2 \|u_{>\frac{N}{8}}\|_{L_{t,x}^4} \lesssim \eta,
\end{align*} 
uniformly in $\lambda>0$, as desired. 

We may now estimate the contribution of \eqref{high} to \eqref{WTS}. Using the shifted Strichartz estimate, scaling, and \eqref{high-control1}, we have
\begin{align*}
\biggl\| &\int_0^t e^{i(t+\theta-s)\Delta}[F(S_\lambda u)-F(S_\lambda u_{\leq N})]\,ds\biggr\|_{L_{\theta}^\infty L_{t,x}^4} \\
& \lesssim \|F(S_\lambda u)-F(S_\lambda u_{\leq N})\|_{L_{t,x}^{\frac43}} \\
& \lesssim \bigl\{\|S_\lambda u\|_{L_{t,x}^4}^2 + \|S_\lambda u_{\leq N}\|_{L_{t,x}^4}^2\bigr\}\|S_{\lambda} u_{>N}\|_{L_{t,x}^4} \\
& \lesssim \|u\|_{L_{t,x}^4}^2 \|u_{>N}\|_{L_{t,x}^4}  \lesssim \eta  
\end{align*}
uniformly in $\lambda>0$. 

We turn to the contribution of \eqref{dm-high} to \eqref{WTS}. Using the definition of $F_{DM}$, the shifted Strichartz estimate, \eqref{av-bd}, and \eqref{high-control3}, we have
\begin{align*}
\biggl\| &\int_0^t e^{i(t+\theta-s)\Delta}[F_{DM}(S_\lambda u)-F_{DM}(S_\lambda u_{\leq N})]\,ds\biggr\|_{L_{\theta}^\infty L_{t,x}^4} \\
& \lesssim \int_0^1 \biggl\| \int_0^t e^{i(t-s+\theta-\sigma)\Delta}[F(e^{i\sigma\Delta}S_{\lambda} u)-F(e^{i\sigma\Delta}S_{\lambda} u_{\leq N})]\,ds\biggr\|_{L_{\theta}^\infty L_{t,x}^4}\,d\sigma \\
& \lesssim \int_0^1 \|F(e^{i\sigma\Delta}S_{\lambda} u)-F(e^{i\sigma\Delta}S_{\lambda} u_{\leq N})\|_{L_{t,x}^{\frac43}}\,d\sigma \\
& \lesssim \bigl\{ \|e^{i\sigma\Delta} S_\lambda u\|_{L_{\sigma}^\infty L_{t,x}^4}^2 + \|e^{i\sigma\Delta} S_{\lambda}u_{\leq N}\|_{L_{\sigma}^\infty L_{t,x}^4}^2\bigr\}\|e^{i\sigma\Delta} S_{\lambda}u_{>N}\|_{L_{\sigma}^\infty L_{t,x}^4} \\
& \lesssim \eta
\end{align*}
uniformly in $\lambda>0$. 

It remains to estimate the contribution of \eqref{low} to \eqref{WTS}. Let us introduce the notation 
\[
H(\tau)= H_{\lambda,N}(\tau;t,x):=e^{-i\tau\Delta}\bigl[F(e^{i\tau\Delta}S_\lambda u_{\leq N})\bigr]. 
\]
In particular, we may write
\begin{align*}
F_{DM}(S_\lambda u_{\leq N}) - F(S_\lambda u_{\leq N}) & =  \int_0^1 H(\sigma)\,d\sigma - H(0) \\
& = \int_0^1 \int_0^\sigma \partial_\tau H(\tau)\,d\tau\,d\sigma. 
\end{align*}

Writing 
\[
\psi = e^{i\tau\Delta}S_\lambda u_{\leq N},
\]
a direct computation now yields
\begin{align*}
\partial_\tau H(\tau)  & = -2i e^{-i\tau\Delta}[ \bar\psi(\nabla\psi\cdot\nabla\psi)+2|\nabla\psi|^2\psi + \psi^2 \Delta\bar\psi] \\
& = \lambda^2\O\bigl\{ e^{-i\tau\Delta}[ (e^{i\tau\Delta}S_\lambda u_{\leq N})(e^{i\tau\Delta}S_\lambda\nabla u_{\leq N})^2 \\
& \quad\quad\quad\quad + (e^{i\tau\Delta}S_\lambda u_{\leq N})^2 e^{i\tau\Delta}S_\lambda \Delta u_{\leq N}]\bigr\},
\end{align*}
where we have used the identities
\[
\nabla [S_\lambda g] = \lambda S_\lambda\nabla g,\quad \Delta[S_\lambda g]=\lambda^2 S_\lambda\Delta g
\]
and again employed the $\O$ notation. Thus we obtain
\begin{align}
\biggl\| & \int_0^t e^{i(t+\theta-s)\Delta}[F_{DM}(S_\lambda u_{\leq N}) - F(S_\lambda u_{\leq N})]\,ds\biggr\|_{L_{\theta}^\infty L_{t,x}^4} \nonumber \\
& \lesssim \lambda^2\int_0^1 \int_0^\sigma \biggl\| e^{i(t-s+\theta-\tau)\Delta}[(e^{i\tau\Delta}S_\lambda u_{\leq N})(e^{i\tau\Delta}S_\lambda\nabla u_{\leq N})^2]\,ds\biggr\|_{L_{\theta}^\infty L_{t,x}^4}\,d\tau\,d\sigma \label{low1} \\
&\quad +  \lambda^2\int_0^1 \int_0^\sigma \biggl\| e^{i(t-s+\theta-\tau)\Delta}[ (e^{i\tau\Delta}S_\lambda u_{\leq N})^2 e^{i\tau\Delta}S_\lambda \Delta u_{\leq N}]\,ds\biggr\|_{L_{\theta}^\infty L_{t,x}^4}\,d\tau\,d\sigma. \label{low2}
\end{align}
By the shifted Strichartz estimate, we have
\begin{align*}
|\eqref{low1}| & \lesssim \lambda^2 \int_0^1 \int_0^1 \|e^{i\tau\Delta}S_\lambda u_{\leq N}\|_{L_{t,x}^4} \|e^{i\tau\Delta}S_\lambda \nabla u_{\leq N}\|_{L_{t,x}^4}^2 \,d\tau\,d\sigma \\
& \lesssim \lambda^2 \|e^{i\tau\Delta}S_\lambda u_{\leq N}\|_{L_{\tau}^\infty L_{t,x}^4} \|e^{i\tau\Delta}S_{\lambda}\nabla u_{\leq N}\|_{L_{\tau}^\infty L_{t,x}^4}^2. 
\end{align*}
Now we claim that
\[
\|e^{i\tau\Delta}S_\lambda \nabla u_{\leq N}\|_{L_{\tau}^\infty L_{t,x}^4} \lesssim N,\qtq{uniformly in}\lambda>0.
\]
Indeed, we can proceed as above, first writing
\[
\|e^{i\tau\Delta}S_\lambda \nabla u_{\leq N}\|_{L_{\tau}^\infty L_{t,x}^4}=\|e^{i\tau\Delta}\nabla u_{\leq N}\|_{L_{\tau}^\infty L_{t,x}^4([0,\lambda^2]\times\R\times\R^2)}.
\]
We then write
\[
e^{i\tau\Delta} \nabla u_{\leq N} = e^{i(t+\tau)\Delta}\nabla\varphi_{\leq N} + \nabla P_{\leq N}\int_0^{t} e^{i(t+\tau-s)\Delta}F(u_{\leq N})\,ds.
\]
We can then estimate as above, controlling the effect of $\nabla P_{\leq N}$ by $N$ via Bernstein estimates. Similar reasoning leads to
\[
\|e^{i\tau\Delta}S_\lambda \Delta u_{\leq N}\|_{L_{\tau}^\infty L_{t,x}^4}\lesssim N^2,\qtq{uniformly in}\lambda>0. 
\]

Thus, continuing from above, we obtain
\[
|\eqref{low1}| + |\eqref{low2}| \lesssim \lambda^2 N^2. 
\]

Collecting the estimates for the contributions of \eqref{low}, \eqref{high}, and \eqref{dm-high}, we derive
\begin{equation}\label{all_together_bound}
\biggl\| \int_0^t e^{i(t+\theta-s)\Delta}[(i\partial_s + \Delta)u_\lambda - F_{DM}(u_\lambda) ]\,ds\biggr\|_{L_{\theta}^\infty L_{t,x}^4([0,1]\times\R\times\R^2)} \lesssim \eta + \lambda^2 N^2.
\end{equation}
In particular, choosing $N=N(\eta)$ large and then $\lambda\leq \eta^{\frac12} N^{-1}$, we deduce that the left-hand side is controlled by $\eta$.  In particular, the left-hand side tends to zero as $\lambda\to 0$, which is \eqref{WTS}.

It follows that for all $\lambda$ sufficiently small, we may apply the stability result (Theorem~\ref{T:stability}) to deduce the existence of scattering solutions $v_\lambda$ to \eqref{dmnls} with initial data given by $\varphi_\lambda$.  Moreover, an application of this theorem implies that for any $\eta>0$, we have
\[
\|e^{i\theta\Delta}[u_{\lambda}-v_{\lambda}]\|_{L_{\theta}^\infty L_{t,x}^4} \lesssim \text{LHS}\eqref{all_together_bound} \lesssim \eta
\]
for all $\lambda$ sufficiently small.  That is,
\[
\lim_{\lambda\to 0} \|e^{i\theta\Delta}[u_{\lambda}-v_{\lambda}]\|_{L_{\theta}^\infty L_{t,x}^4} = 0. 
\]
Incorporating the $L_t^\infty L_x^2$-norm into the analysis above, we complete the proof.\end{proof}

\end{document}